\documentclass[12pt,twoside]{amsart}
\usepackage{amsmath,amssymb,amsfonts,amsthm,amsopn}
\usepackage{graphics}

\def\epsilon{\varepsilon}

\newtheorem{theorem}{Theorem}[section]
\newtheorem{lemma}[theorem]{Lemma}
\newtheorem{corollary}[theorem]{Corollary}
\newtheorem{proposition}[theorem]{Proposition}

\newtheorem{remark}[theorem]{Remark}

\newcommand{\beqa}{\begin{eqnarray*}}
\newcommand{\eeqa}{\end{eqnarray*}}

\newcommand{\field}[1]{\mathbb{#1}}
\newcommand{\bR}{\field{R}}        
\newcommand{\bN}{\field{N}}        
\newcommand{\bZ}{\field{Z}}        
\def\cS{\mathcal{S}}

\def\rd{\bR^d}

\def\rdd{{\bR^{2d}}}

\def\R{\right)}

\def\<{\left<}
\def\>{\right>}

\def\mv1{M_v^1}

\def\mn{(m,n)}
\def\mn'{(m',n')}

\def\R{\mathbb{R}}
\def\Ren{\mathbb{R}^d}

\def\Fur{\mathcal{F}}

\def\Sn2{S_{2}(L^{2}(\Ren))}
\def\S1{S_{1}(L^{2}(\Ren))}
\def\sig00{\sigma_{0,0}}

\begin{document}
\begin{abstract}
We present a wave packet analysis of a class of possibly degenerate parabolic equations with variable coefficients. As a consequence, we prove local wellposedness of the corresponding Cauchy problem in spaces of low regularity, namely the modulation spaces, assuming a nonlinearity of analytic type. As another application, we deduce that the corresponding phase space flow decreases the global wave front set. We also consider the action on spaces of analytic functions, provided the coefficients are analytic themselves.
\end{abstract}

\title[Phase space analysis of semilinear parabolic equations]{Phase space analysis of semilinear parabolic equations}

\author{Fabio Nicola}
\address{Dipartimento di Scienze Matematiche,
Politecnico di Torino, corso Duca degli Abruzzi 24, 10129 Torino,
Italy}
\email{fabio.nicola@polito.it}


\maketitle

\section{Introduction}
The study of the local and global wellposedness of nonlinear evolution PDEs in spaces of low regularity represents one the most active research fields, where the deepest machinery of modern Harmonic Analysis is applied \cite{tao}. In particular, a number of results were obtained for the Schr\"odinger, wave, Klein-Gordon and Navier-Stokes equations with initial data in the so-called modulation spaces \cite{bertinoro3,bertinoro2,cnjde,bertinoro12,cnr2,kki1,kki2,kki3,MNRTT,baoxiang0,bertinoro58,bertinoro58bis,bertinoro57}. These spaces were used as a fundamental tool in Time-frequency Analysis \cite{fei,book} but their role in PDEs has been recognized only recently. We refer to the recent survey \cite{ruz} and monograph \cite{baoxiang} for a detailed overview of results and techniques. Variable coefficient Schr\"odinger equations were considered in \cite{cgnr2,cn,cnr1,kki4,tataru}. \par 
Modulation spaces can be defined similarly to the Besov spaces, but for a different geometry: the dyadic annuli in the frequency domain are replaced by isometric boxes $\mathcal{Q}_k$, $k\in\bZ^d$, which allows a finer analysis in many respects. Namely, for $1\leq p,q\leq\infty$, $s\geq 0$, one defines
\begin{equation}\label{prima}
M^{p,q}_s=\Big\{f\in\cS'(\rd): \|f\|_{M^{p,q}_s}:=\Big(\sum_{k\in\bZ^d}\langle k\rangle^{sq}\|\square_k f\|_{L^p}^q\Big)^{1/q}<\infty\Big\}
\end{equation}
(with obvious changes if $q=\infty$), where $\square_k$ are Fourier multipliers with symbols $\chi_{\mathcal{Q}_k}$ conveniently smoothed; we also set $M^{p,q}$ for $M^{p,q}_0$. We recapture in particular the $L^2$-based Sobolev spaces $M^{2,2}_s=H^s$, whereas the space $M^{\infty,1}$ coincides with the so-called Sj\"ostrand's class \cite{lerner,sjostrand}. \par
Here we study the local wellposedness of the Cauchy problem in modulation spaces of a wide class of parabolic, possibly degenerate, semilinear equations in $\rd$.\par We consider the symbol classes $S^{(k)}_{0,0}$, $k\in\bN$ \cite{tataru}, defined by the estimates
\begin{equation}\label{s200}
|\partial^\alpha_{\xi}\partial^\beta_x a(x,\xi)|\leq C_{\alpha,\beta},\quad |\alpha|+|\beta|\geq k,\quad x,\xi\in\rd,
\end{equation}
endowed with the obvious seminorms. We also allow the symbols to be time-dependent.\par
Namely, let $T>0$ be fixed and consider the pseudodifferential operator (Weyl quantization)
\begin{equation}\label{equazione}
L=\partial_t+a^w(t,x,D)+ib^w(t,x,D),
\end{equation}
where the symbols $a(t,x,\xi)$ and $b(t,x,\xi)$ of the diffusion and drift terms are {\it real-valued} and satisfy the following conditions:
\begin{itemize}
\item[\bf (i)] $a(t,\cdot)$ belongs to a bounded subset of $S^{(2)}_{0,0}$ for $t\in[0,T]$;
\item[\bf (ii)]$a(t,x,\xi)\geq -C$ for some constant $C>0$ and every $t\in[0,T]$, $x,\xi\in\rd$;
\item[\bf (iii)] $b(t,\cdot)$ belongs to a bounded subset of $S^{(1)}_{0,0}$ for $t\in[0,T]$;
\item[\bf (iv)] the maps $t\mapsto a(t,\cdot)$, $t\mapsto b(t,\cdot)$ are (weakly) continuous from $[0,T]$ to $S'(\rd)$ (or equivalently pointwise).
\end{itemize}
Observe that we do not assume any ellipticity-type condition. As a very simple example, one may consider the operator (cf.\ \cite{33,187})
\[
\partial_t-\sum_{j=1}^d a_j(t)\partial^2_{x_j}+\sum_{j=1}^d\tilde{a}_j(t)\partial_{x_j}+V_1(t,x)+iV_2(t,x),
\]
where $a_j,\tilde{a}_j \in C([0,T])$ are real-valued, with $a_j(t)\geq0$, $j=1,\ldots,d$, and $V_1$ and $V_2$ are real-valued, continuous with respect to $t\in[0,T]$ for fixed $x$, with $V_1(t,x)\geq -C$ and satisfying for $t\in[0,T]$, $x\in\rd$,
\[
|\partial^\alpha_x V_1(t,x)|\leq C_\alpha,\ |\alpha|\geq 2; \quad |\partial^\alpha_x V_2(t,x)|\leq C_\alpha,\ |\alpha|\geq 1. 
\]  
We further consider a nonlinearity of the form 
\begin{align}\label{nonlinea} 
\textit{$\mathcal{N}(t,x,u)=g(t,x) F(u)$ with $g\in C([0,T]; M^{\infty,1}_s)$ for some $s\geq0$,}\\
\textit{where the function $F$ is entire real-analytic, with $F(0)=0$}\nonumber
\end{align}
($F(z)$ has a Taylor expansion in $z,\overline{z}$, valid in the whole complex plane). \par
In particular we can take a polynomial in $u,\overline{u}$. The case $\mathcal{N}(t,x,u)=g(t,x)u$, with $g\in C([0,T]; M^{\infty,1})$,  corresponds to a potential in the Sj\"ostrand's class. \par
We have therefore the following local wellposedness result.   
\begin{theorem}\label{mainteo0} 
Assume the above hypotheses {\bf (i)--(iv)} and \eqref{nonlinea}. Let $1\leq p<\infty$. There exists $T_0\leq T$ such that for every $u_0\in M^{p,1}_s$ there exists a unique solution $u\in C([0,T_0];M^{p,1}_s)$ to the Cauchy problem
\begin{equation}\label{equazione0} 
\begin{cases}
Lu=\mathcal{N}(t,x,u)\\
u(0)=u_0.
\end{cases}
\end{equation} 
Moreover, the map $M^{p,1}_s \ni u_0\mapsto u\in C([0,T_0];M^{p,1}_s)$ is Lipschitz continuous on every ball. \par
The same result holds for $p=\infty$ if one replaces $M^{\infty,1}_s$ with the closure $\tilde{M}^{\infty,1}_s$ of the Schwartz space in $M^{\infty,1}_s$.\par

\end{theorem}
The above result is mainly of interest for the space $\tilde{M}^{\infty,1}_0$, since it contains the spaces $M^{p,1}_s$ for every $p<\infty$, $s\geq 0$. For comparison with the Besov spaces, we recall the sharp embedding $B^{d}_{\infty,1}\subset M^{\infty,1}\subset B^0_{\infty,1}$ \cite{baoxiang}. In general, elements in $\tilde{M}^{\infty,1}_0$ do not possess any fractional derivative in $L^2_{loc}$. \par\par
This result is inspired by the detailed analysis in \cite{bertinoro57,baoxiang} for the Ginzburg-Landau operator $\partial_t-(a+i)\Delta$, $a\geq0$. Contrary to that situation, here we do not have wellposedness in the whole space $M^{\infty,1}$ and our result does not extend to second order symbols $b$ in \eqref{equazione}; see the counterexamples in Remarks \ref{contro1}, \ref{contro2} below. Moreover, for $a>0$ the Ginzburg-Landau semigroup displays a smoothing effect of infinite order once $t>0$, which forces the solution to be analytic for initial data in the above modulation spaces (cf.\ \cite{clx1,clx2,clx3,clx4,lerner2,morimoto,villani}). In the present situation, where degeneracy is allowed, we do not have this phenomenon and the above result seems therefore optimal.  Actually, we will prove in Section 6 that the linear propagator, although does not increase regularity, still preserves the space of analytic functions if the coefficients are analytic themselves. 
\par
The proof of Theorem \ref{mainteo0} relies on a representation of the linear propagator as a pseudodifferential operator with symbol in the class $S^0_{0,0}=S^{(0)}_{0,0}$. To be precise we have the following key result.  
\begin{theorem}\label{mainteo} 
Under the above assumptions, for every $u_0\in\cS(\rd)$ there exists a unique solution $u\in C^1([0,T];\cS(\rd))$ to $Lu=0$, $u(0)=u_0$, and the propagator $S(t): u_0\mapsto u(t)$ is a pseudodifferential operator whose Weyl symbol $p(t,x,\xi)$ lies in a bounded subset of $S^0_{0,0}$ for $t\in[0,T]$. \par
Moreover the symbol $p$ is continuous as a function of $t\in[0,T]$ valued in $\cS'(\rdd)$.
\end{theorem}
This is a classical result if $a(t,x,\xi)$ is elliptic, of arbitrary order, in the framework of H\"ormander's symbol classes of type $\rho=1$, $\delta=0$ \cite{treves}. Similar results were obtained by several authors under assumptions of subellipticity \cite{cinesi}. Remarkably, the representation as a pseudodifferential operator still holds in the present generality - although in the exotic symbol class $S^0_{0,0}$. The lack of a full symbolic calculus prevents us to follow the classical proof. Instead, we will use a characterization of pseudodifferential operators contained in the seminal papers \cite{g-ibero,tataru}, which involves their (continuous) matrix representation in phase space, with respect to Gabor wave packets (see also \cite{GR} for further generalizations to other classes of rough symbols, and \cite{cgnr,cnr1} for similar results for Fourier integral operators). Basically we will see that the desired representation of the propagator is a consequence of the wellposedness in $\cS(\rd)$ of the linear Cauchy problem, provided a certain uniformity with respect to phase space shifts holds (in a sense that will be made clear below; see \eqref{pri}). This principle seems widely applicable (conveniently modified) to other situations where a global symbolic calculus is not available and the fundamental solution cannot therefore be constructed via asymptotic expansions. This occurs, for example, for classes of rough symbols, or even in the analytic category.  We plan to investigate these issues elsewhere.\par
Finally we present in Section \ref{applicazioni} another application of Theorem \ref{mainteo}. Namely, the linear propagator for operators of the form \eqref{equazione} generally does not preserve the classical wave front set; consider, e.g., the case $a\equiv0$, $b(t,x,\xi)=\xi$, in dimension $d=1$, for which $S(t)$ is a translation. Besides this trivial example, we refer to \cite{parenti} for sharp propagation results of the classical wave front set for operators as in \eqref{equazione}. On the other hand, it was proved in \cite{rw} that pseudodifferential operators with symbols in $S^0_{0,0}$ preserve a type of {\it global} wave front set $WF_G$ introduced by H\"ormander in \cite{hormander} and defined in terms of conic subsets of $\rdd$ (rather than $\R^d$). This variant is suitable in dealing with operators in the whole Euclidean space, since it measures both regularity and decay of solutions; for example, if $f\in\cS'(\rd)$ then $f\in\cS(\rd)$ if and only if $WF_G(f)=\emptyset$.  Therefore, we see that the linear propagator $S(t)$ decreases such a wave a front set, i.e.\ it is {\it globally pseudolocal}.\par\medskip
Briefly, the paper is organized as follows. Section 2 is devoted to basic notation whereas in Section 3 we show a preliminary lower bound. Sections 4 and 5 are devoted to the proofs of Theorems \ref{mainteo} and \ref{mainteo0} respectively. In Section 6 we study the action of the linear propagator on the space of analytic functions, whereas in Section 7 we prove its global pseudolocality. 
\section{Notation}
For $z=(x,\xi)\in\rd\times\rd$ we define the phase space shifts $\pi(z)f=M_\xi T_x f$, where $T_x f(y)=f(y-x)$ and $M_\xi f(y)=e^{i\xi y} f(y)$ are the translation and modulation operators. Notice that $\pi(z)$ is a unitary operator. We will denote by $\langle\cdot,\cdot\rangle$ the inner product in $L^2(\rd)$, or the pairing between $\cS'(\rd)$ and $\cS(\rd)$ (anti-linear on the second factor), and by $||\cdot||$ the $L^2$-norm. The Fourier transform is normalized as $$\Fur(f)(\xi)=\widehat{f}(\xi)=\int e^{-ix\xi} f(x)\,dx,$$ and the Weyl quantization of a symbol $a(x,\xi)$ is correspondingly defined as 
\[
a^w(x,D)f=(2\pi)^{-d}\iint e^{i(x-y)\xi}a\big(\frac{x+y}{2},\xi\big)f(y)\,dy\,d\xi.
\]
We recall that real-valued symbols give rise to formally self-adjoint operators.\par
As usual we denote by $S^0_{0,0}$ the space of smooth functions in $\rdd$ which are bounded together with their derivatives of every order. The symbol classes $S^{(k)}_{0,0}$, $k\in\bN$, were already introduced above in \eqref{s200}.\par
Given $f\in\cS'(\rd)$, $g\in\cS(\rd)$, we define the short-time Fourier transform of $f$ with window $g$ as 
\[
V_g f(z)=\langle f,\pi(z)g\rangle,\quad z=(x,\xi)\in\rd\times\rd.
\]
We anticipated in the introduction (cf.\ \eqref{prima}) the definition of the modulation spaces $M^{p,q}_s$ in terms of a uniform decomposition of the frequency domain. In the sequel we will use some boundedness results from papers  where modulation spaces are defined in terms of the short-time Fourier transform. For the sake of completeness we recall that in fact we have
\[
M^{p,q}_s=\Big\{f\in\cS'(\rd): \int\Big(\int \langle \xi\rangle^{ps}|V_g f(x,\xi)|^p\,dx\Big)^{q/p}\,d\xi\Big)^{1/q}<\infty\Big\},
\]
where $g$ is any non-zero Schwartz function (with obvious changes if $p=\infty$ or $q=\infty$). The equivalence of the two definitions is shown e.g.\ in \cite{bertinoro58bis}.\par
We refer to \cite[Chapter 11]{book} and \cite[Chapter 6]{baoxiang} for properties and applications of modulation spaces to Time-frequency Analysis and PDEs respectively.
 \par
 \section{A preliminary lower bound}
We will need the following lower bound. It is basically a consequence of the sharp G\a r arding inequality but it is a bit subtler for the required uniformity of the constants. 
\begin{lemma}\label{lemmachiave}
Let $a(x,\xi)$ and $b(x,\xi)$ be real-valued symbols in $S^{(2)}_{0,0}$ and $S^{(1)}_{0,0}$ respectively, with $a(t,x,\xi)\geq -C_0$ for some constant $C_0>0$ and every $x,\xi\in\rd$. \par Let 
$$\mathcal{L}=a^w(x,D)+i b^w(x,D)$$
 and, for $k\in\bZ$, $E_k(D)=(1-\Delta)^k$, $E_k(x)=(1+|x|^2)^k$. Then for every $k\in\bZ$ there exists a constant $C>0$ depending only on $k$, the above constant $C_0$ and some seminorm of $a$ and $b$ in $S^{(2)}_{0,0}$ and $S^{(1)}_{0,0}$ respectively (hence on upper bounds for $|\partial^\alpha_x\partial^\beta_\xi a|$, $|\alpha|+|\beta|\geq 2$ and $|\partial^\alpha_x\partial^\beta_\xi b|$, $|\alpha|+|\beta|\geq 1$) such that, for every $u\in \cS(\rd)$, 
\begin{equation}\label{en1}
{\rm Re} \langle E_k(D) \mathcal{L}u,E_k(D) u\rangle \geq -C\|E_k(D)u\|^2
\end{equation}
\begin{equation}\label{en2}
{\rm Re} \langle E_k(x) \mathcal{L}u,E_k(x) u\rangle \geq -C\|E_k(x)u\|^2.
\end{equation}
\end{lemma}
\begin{proof}
We will prove only \eqref{en1}, because \eqref{en2} follows quickly by Plancherel's formula if one applies \eqref{en1} to the operator with Weyl symbol $a(\xi,-x)+ib(\xi,-x)$ and conjugates with the Fourier transform (cf.\ \cite[Theorem 18.5.9]{hormanderIII}).  \par
Now, we have 
\[
{\rm Re} \langle E_k(D) \mathcal{L} u,E_k(D)  u\rangle={\rm Re} \langle E_k(D) \mathcal{L}E_{-k}(D)  E_k(D)  u,E_k(D)  u\rangle,
\]
so that it suffices to prove that the operators $$A:= E_k(D)  a^w(x,D) E_{-k}(D),\quad B:= E_k(D)  b^w(x,D)E_{-k}(D) $$ satisfy the estimates
\begin{equation}\label{en5}
{\rm Re}\langle A u,u\rangle\geq-C\|u\|^2,\quad |{\rm Re}\langle iB u,u\rangle|\leq C\|u\|^2
\end{equation}
for some constant $C>0$ as in the statement. \par Let us prove the estimate for $A$. By the  pseudodifferential calculus \cite[Chapter XVIII]{hormanderIII}, $a^w(x,D)E_{-k}(D) $ has Weyl symbol 
\begin{equation}\label{ag0}
a(x,\xi)(1+|\xi|^2)^{-k}-ki\sum_{j=1}^d\xi_j (1+|\xi|^2)^{-k-1} \partial_{x_j}a(x,\xi)+r(x,\xi),
\end{equation}
where the symbol $(1+|\xi|^2)^{k+2} r(x,\xi)$  belongs to $S^0_{0,0}$, with seminorms dominated by those of $a$ in $S^{(2)}_{0,0}$ (only the derivatives of order $\geq 2$ of $a(x,\xi)$ are involved in the expression of the remainder $r$). Observe that the sum in \eqref{ag0} gives rise to a $0$-order symbol too, but it cannot be neglected, because it involves first derivatives of $a(x,\xi)$. Similarly, by the symbolic calculus one sees that the Weyl symbol of $A$ is given by 
\begin{equation}\label{en4}
a(x,\xi)-2ki\sum_{j=1}^d\xi_j (1+|\xi|^2)^{-1} \partial_{x_j}a(x,\xi)+r'(x,\xi),
\end{equation}
with $r'(x,\xi)\in S^0_{0,0}$, again with seminorms depending only on those of $a$ in $S^{(2)}_{0,0}$; the key fact is that, except for the terms in the above sum, the derivatives of $a(x,\xi)$ which arise in the symbolic calculus have order $\geq2$ (in turn this comes from the fact that the symbol of $(1-\Delta)^k$ is independent of $x$). Now, the second term in \eqref{en4} gives rise to a skew-adjoint operator in $L^2$, so that the first formula in \eqref{en5} follows from the sharp G\a r arding inequality (see e.g.\ \cite[Theorem 2.5.15]{lerner} for a version where the lower bound constant is shown to depend only on the seminorms in $S^0_{0,0}$ of the second derivatives of the symbol). \par The second formula in \eqref{en5} is shown by a similar and easier argument.
\end{proof}

\section{The linear propagator: proof of Theorem \ref{mainteo}}
First of all we show that the weak continuity of the symbol $p$ as a function of $t$ follows from the first part of the statement. In fact, we have equivalently to prove that the Schwartz kernel of $S(t)$, or equivalently that of the adjoint $S(t)^\ast$ depends continuously on $t$ for the (weak) topology of $\cS'(\rdd)$. Now, since $S(t)^\ast$ acts from $\cS'(\rd)$ into itself, its kernel $K_t(x,y)$ is given by $S(t)^\ast [\delta(\cdot-y)]$; namely, if $\varphi(x,y)$ is a function in $\cS(\rdd)$,
\[
\langle K_t, \varphi\rangle=\langle S(t)^\ast [\delta(x-y)],\varphi(x,y)\rangle=\langle \delta(x-y),S(t)\varphi(x,y)\rangle 
\]
(where $S(t)$ and its adjoint act on the variable $x$). On the other hand, this last expression is continuous as a function of $t$, because $S(t)$ is strongly continuous on $\cS(\rd)$. \par
We now come to the first part of the statement. Once we get existence and uniqueness of the solution, the desired representation of the propagator $S(t)$ is obtained by the following characterization of pseudodifferental operators with symbols in $S^0_{0,0}$ via their phase space kernel (\cite[Theorem 3.2]{g-ibero}, \cite[Theorem 1]{tataru}): 
\begin{theorem}\label{teocar}
Let $A$ be a linear continuous operator $\cS(\rd)\to\cS'(\rd)$. Then $A$ is a pseudodifferential operator with Weyl symbol in $S^0_{0,0}$ if and only if for every $N\in\bN$ and for some (and therefore every) Schwartz function $g\not\equiv0$, there exists a constant $C_N>0$ such that   
\[
|\langle A\pi(z)g,\pi(w)g\rangle|\leq  C_N(1+|w-z|)^{-N}\quad\forall z,w\in\rdd.
\]
In fact, each seminorm of the symbol in $S^0_{0,0}$ is estimated by the above constant $C_N$ for some $N>0$.
\end{theorem} 
Hence we have to verify that for any given non-zero Schwartz window $g$, it turns out
\begin{equation}\label{ve1}
|\langle S(t)\pi(z)g,\pi(w)g\rangle|\leq  C_N(1+|w-z|)^{-N}\quad\forall z,w\in\rdd,\ t\in[0,T],
\end{equation}
for every $N\in\bN$. 
This is equivalent to 
\begin{equation}\label{ve2}
|\langle \pi(z)^\ast S(t)\pi(z)g,\pi(w-z)g\rangle|\leq  C_N(1+|w-z|)^{-N} \quad\forall z,w\in\rdd,\ t\in[0,T].
\end{equation}
Now we will make use of the following remark.
\begin{lemma}
Let $f,g$ be Schwartz functions. Then for every $N\in\bN$ there exist $k\in\bN$ and $C_N>0$ such that
\begin{equation}\label{ve3}
|\langle f,\pi(w)g\rangle|\leq  C_N|f|_k|g|_k(1+|w|)^{-N}
\end{equation}
where we set $|f|_k=\sup\limits_{|\alpha|+|\beta|\leq k}\sup\limits_{x\in\rd}|x^\beta\partial^\alpha f(x)|$. 
\end{lemma}
\begin{proof}
We recognize in the left-hand side of \eqref{ve3} the short-time Fourier transform $V_g f(w)=\langle f,\pi(w)g\rangle$ of the function $f$ with window $g$. Now, it is well-known that the map $(g,f)\mapsto V_g f$ is continuous from $\cS(\rd)\times\cS(\rd)$ to $\cS(\rdd)$ (see e.g.\ \cite[Theorem 11.2.5]{book}). 
\end{proof}

In view of this lemma, \eqref{ve2} will follow if we prove that\par\medskip {\it The function $\pi(z)^\ast S(t)\pi(z)g$ is Schwartz, with seminorms uniformly bounded with respect to $z\in\rdd$ and $t\in[0,T]$}.\medskip \\
Now, let $\mathcal{L}=a^w(t,x,D)+ib^w(t,x,D)$. Then $u=\pi(z)^\ast S(t)\pi(z)g$ satisfies
\begin{equation}\label{ve4}
\begin{cases}\partial_t u+\mathcal{L}_z u=0\\
u(0)=g,
\end{cases}
\end{equation}
where  
$$\mathcal{L}_z:=\pi(z)^\ast \mathcal{L}\pi(z)=a_z^w(t,x,D)+ib_z^w(t,x,D)$$
with
\[
a_z(t,x,\xi)=a(t,x+x_0,\xi+\xi_0),\quad b_z(t,x,\xi)=b(t,x+x_0,x+\xi_0) ,\ {\rm for}\ z=(x_0,\xi_0);
\]
(cf.\ the symplectic invariance of the Weyl quantization in \cite[Theorem 18.5.9]{hormanderIII}). Observe that $a_z$ and $b_z$ belong to bounded subsets of $S^{(2)}_{0,0}$ and $S^{(1)}_{0,0}$  respectively, for $z\in\rdd$. \par Hence we are left to prove the following result.
\begin{align}\label{pri} 
\textit{For every fixed $g\in\cS(\rd)$ the Cauchy problem \eqref{ve4} has a unique solution}\\
\textit{$u\in C^1([0,T];\cS(\rd))$, with seminorms uniformly bounded with respect to $z\in\rdd$.}\nonumber\end{align}
\noindent In particular, for $z=0$ we get the wellposedness in $\cS(\rd)$ of the original problem \eqref{equazione}.\par To this end we need some energy estimates. We define the weighted Sobolev spaces $Q^{2k}$, $k\in\bZ$ (the case of even integer exponents will suffice), as follows:
\begin{equation}\label{shubin}
Q^{2k}:=\big\{f\in L^2(\rd):\ \|f\|^2_{Q^{2k}}:=\|f\|_{H^{2k}}^2+\|\widehat{f}\|^2_{H^{2k}}<\infty\big\},
\end{equation}
where $H^s$ denotes the usual Sobolev space in $\rd$. Observe that $\cS(\rd)=\cap_{k\in\bZ} Q^{2k}$ topologically (and $\cS'(\rd)=\cup_{k\in\bZ} Q^{2k}$); see e.g.\ \cite{nr}. Hence  it remains to prove the following result.  
\begin{proposition}\label{proposizione} Let 
$$L_z=\partial_t+\mathcal{L}_z=\partial_t+a_z^w(t,x,D)+ib_z^w(t,x,D).$$
 For every $k\in\bZ$ there exists a constant $C$, depending only on $k$ and on some seminorm of $a$ and $b$ in $S^{(2)}_{0,0}$ and $S^{(1)}_{0,0}$ respectively, such that for every function $u\in C^1([0,T]; Q^{2k}(\rd))\cap C^0([0,T];Q^{2k+2})$ we have 
\begin{equation}\label{apriori}
\|u(t)\|_{Q^{2k}}\leq C \|u(0)\|_{Q^{2k}}+C\int_0^t \|L_zu(s)\|_{Q^{2k}}\, ds,\quad \forall t\in[0,T].
\end{equation}
As a consequence, for every initial datum $g\in Q^{2k}$ there exists a unique solution $u\in C^0([0,T]; Q^{2k})\cap C^1([0,T];Q^{2k-2})$ to \eqref{ve4}, with norms uniformly bounded with respect to $z$. 
\end{proposition}
\begin{remark}\rm
The point is the uniformity with respect to $z$ of the constant $C$ in the energy estimate \eqref{apriori}. For example, for the Schr\"odinger equation, \eqref{apriori} holds too, but not uniformly with respect to $z$ (and in fact the corresponding propagator is not a pseudodifferential operator with symbol in $S^0_{0,0}$). 
\end{remark}
\begin{proof}[Proof of Proposition \ref{proposizione}] The argument is classical and therefore we only sketch the proof. We observe, first of all, that for fixed $z$, $a_z^w(t,x,D)$ and $b_z^w(t,x,D)$ by the assumptions {\bf (i),\,(iii),\,(iv)} in the introduction are strongly continuous as operators in $\cS(\rd)$ and also belong to a bounded subset of the space of continuous operators $Q^{2k}\to Q^{2k-2}$, for every $k\in\bZ$ (cf.\ the argument in \cite[pag.\ 386]{hormanderIII} and \cite[Proposition 1.5.5]{nr}). Hence they are also strongly continuous as operators $Q^{2k}\to Q^{2k-2}$.  Let now \[
f=L_z u=\partial_tu+\mathcal{L}_z u=\partial_t u+a_z^w(t,x,D)u+ib_z^w(t,x,D)u.
\]
 We have  
\begin{align*}
\frac{d}{dt}\|u(t)\|^2_{Q^{2k}}&=\frac{d}{dt}\langle u(t),u(t) \rangle_{Q^{2k}}=2{\rm Re}\langle u(t),u'(t) \rangle_{Q^{2k}}\\
&=-2{\rm Re}\langle u(t),\mathcal{L}_zu(t)\rangle_{Q^{2k}}+2{\rm Re}\langle u(t),f(t)u(t)\rangle_{Q^{2k}}
\end{align*}
where we used the equation (i.e.\ the definition of $f$). With the notation in Lemma \ref{lemmachiave}, the first term in the right-hand side reads
\[
-2{\rm Re}\langle E_k(D)  u(t),E_k(D)  \mathcal{L}_z u(t)\rangle
-2{\rm Re}\langle E_k(x)u(t),E_k(x) \mathcal{L}_z u(t)\rangle
\]
and by Lemma \ref{lemmachiave} this expression is $
\leq C\|u\|^2_{Q^{2k}}$ for some constant $C>0$ independent of $z$. \par
Using the Cauchy-Schwarz' inequality in $Q^{2k}$ for the term involving $f$ and then Gronwall's lemma we easily get \eqref{apriori} (cf.\ e.g.\ \cite[Section 2.1.2]{rauch}). \par
The last part of the statement follows from the a priori estimate \eqref{apriori} by an abstract functional analytic argument; see e.g.\ the proof of \cite[Theorem 23.1.2]{hormanderIII}, where one has to replace the space $H^s$ which appear there with our spaces $Q^{2k}$, and observe that the adjoint operator $-\partial_t+a^w_z(t,x,D)-ib_z^w(t,x,D)$ satisfies the same a priori estimates as $L$, but for the backward problem with final condition at $t=T$. 
\end{proof}
\begin{remark}\rm In the statement of Theorem \ref{mainteo} one could equivalently say  that the symbol $p$ of the propagator is continuous as a function of $t$ valued in $C^\infty(\rdd)$ endowed with the usual Frech\'et topology, or even endowed with the pointwise convergence topology; the equivalence is a consequence of the boundedness of the family of symbols $\{p(t,\cdot): t\in[0,T]\}$ in $S^0_{0,0}$; see e.g.\ \cite[Proposition 1.1.2]{nr}.  \par
However, generally speaking, the map $[0,T]\ni t\mapsto p(t,\cdot)$, valued in $S^0_{0,0}$ or even in $L^\infty(\rdd)$, is {\it not} continuous. This is already evident for the classical heat equation ($a(t,x,\xi)=|\xi|^2$, $b\equiv0$).    
\end{remark}
\section{The nonlinear equation: proof of Theorem \ref{mainteo0}}
We use a standard contraction argument; cf.\ \cite[Theorem 6.1]{baoxiang}. We consider first the case $p<\infty$ and argue in the space $X:=C([0,T_0],M^{p,1}_s)$, with $T_0$ small enough. \par We write the semilinear equation in integral form (Duhamel principle) as
\[
u(t)=S(t,0) u_0+\int_0^t S(t,\sigma) \mathcal{N}(\sigma,x,u(\sigma))\,d\sigma
\]
where $S(t,\sigma)$, $0\leq \sigma\leq  t\leq T$, is the linear propagator corresponding to initial data at time $\sigma$.  
The classical iteration scheme works in $X$ if the following properties are verified:
\begin{itemize}
\item[\bf a)] $S(t,\sigma)$ is strongly continuous on $M^{p,1}_s$ for $0\leq \sigma\leq t\leq T$ (which also implies a uniform bound for the operator norm with respect to $\sigma,t$, by the uniform boundedness principle);
\item[\bf b)] The nonlinearity verifies $\|\mathcal{N}(t,x,u)-\mathcal{N}(t,x,v)\|_X\leq C\|u-v\|_X$, for $u,v\in X$ in every fixed ball.
\end{itemize}
Concerning {\bf a)} we observe that, if $u_0\in\cS(\rd)$ then $t\mapsto S(t,\sigma)u_0$ is continuous in $\cS(\rd)$ for fixed $\sigma$, as a consequence of Theorem \ref{mainteo}. \par On the other hand, if $\sigma'\leq \sigma\leq t$ and setting $\mathcal{L}(t)=a^w(t,x,D)+ib^w(t,x,D)$, by \eqref{apriori} applied at the initial time $\sigma$ we have, for every $k\in\bZ$,
\begin{align*}
\|S(t,\sigma)u_0-S(t,\sigma')u_0\|_{Q^{2k}}&\leq C_1\|u_0-S(\sigma,\sigma')u_0\|_{Q^{2k}}\\
&=C_1\|\int_{\sigma'}^\sigma \mathcal{L}(\tau)S(\tau,\sigma')u_0 \,d\tau\|_{Q^{2k}}\leq C_2(\sigma-\sigma')\|u_0\|_{Q^{2k+2}},
\end{align*}
where we used that $\mathcal{L}(t): Q^{2k+2}\to Q^{2k}$ and $S(t,\sigma'):Q^{2k+2}\to Q^{2k+2}$ continuously, with bounds uniform with respect to $t,\sigma'$. 
Hence, the map $\sigma\mapsto S(t,\sigma)u_0$ is continuous in $\cS(\rd)$ uniformly with respect to $t$. This gives the strong continuity of $S(t,\sigma)$ on $\cS(\rd)$, as a function of $\sigma,t$. \par
Finally, we know from Theorem \ref{mainteo} that the operators $S(t,\sigma)$ are pseudodifferential with symbols uniformly bounded in $S^0_{0,0}$. It follows from the continuity results of such operators in modulation spaces (\cite[Theorem 4.1]{g-ibero}) that they belong to a bounded subset of the space of continuous operators on $M^{p,1}_s$. Since $\cS(\rd)$ is dense in $M^{p,1}_s$ we deduce strong continuity on $M^{p,1}_s$.\par
Consider now the point ${\bf b)}$.  The desired estimate was already proved in \cite[Formula (28)]{bertinoro12} for the same nonlinearity as in \eqref{nonlinea}, except for the factor $g(t,x)\in M^{\infty,1}_s$. However one can easily take into account this additional factor, using the fact that $M^{p,1}_s$ is a modulus on $M^{\infty,1}_s$ (\cite[Corollary 4.2]{bertinoro57} and \cite{bertinoro3}). \par
If $p=\infty$ one considers $X:=C([0,T_0],\tilde{M}^{\infty,1}_s)$ and can repeat the same argument as above. This conclude the proof of Theorem \ref{mainteo0}.
\begin{remark}\label{contro1}\rm
We notice that the linear propagator $S(t,\sigma)$, as any pseudodifferential operator with symbol in $S^0_{0,0}$, is still bounded on $M^{\infty,1}$ (\cite[Theorem 4.1]{g-ibero}); however it is not strongly continuous, in general, and the wellposedness may fail already for the linear problem. Consider, for example, the symbols $a(t,x,\xi)=|x|^2$, $b\equiv0$. Then $S(t,0)u_0(x)=e^{-t|x|^2}u_0(x)$, and if $u_0\equiv 1\in M^{\infty,1}$ we have
\[
C\|S(t,0)u_0-u_0\|_{M^{\infty,1}}\geq \|S(t,0)u_0-u_0\|_{L^\infty}=\sup_{x\in\rd}|e^{-t|x|^2}-1|=1
\]
for every $t>0$ and some $C>0$. 
\end{remark}
\begin{remark}\label{contro2}\rm
The results in Theorem \ref{mainteo0} (and Theorem \ref{mainteo}) do not hold if we consider a symbol $b$ in \eqref{equazione} of order $2$. For example, for $b(t,x,\xi)=|x|^2$, $a\equiv0$, we get the linear propagator $S(t,0)u_0(x)=e^{-it|x|^2}u_0(x)$, and when $t>0$ this operator is bounded on $M^{p,q}$ only if $p=q$ (\cite[Proposition 7.1]{cnr1}). 
\end{remark}

\section{The action on the space of analytic functions} 
Consider the space $\mathcal{A}$ of smooth functions $f$ in $\rd$ such that 
\[
||\partial^\alpha f||\leq C^{|\alpha|+1}\alpha!,\quad \alpha\in\bN^d,
\]
for some constant $C>0$.\par
The study of the heat semigroup $e^{t\Delta}$ on this function space was detailed e.g. in \cite{bertinoro57} and \cite[Section 2.5]{baoxiang}, where $\mathcal{A}$ was regarded as the union of modulation spaces with an exponential weight in the frequency variable.  
We now show that for the operator $L$ in \eqref{equazione}, the linear propagator $S(t)=S(t,0)$ still preserves $\mathcal{A}$ provided the symbols $a(t,x,\xi)$ and $b(t,x,\xi)$ are analytic with respect to $x$. Namely, we replace the assumptions {\bf (i),(iii)} in the introduction by the following ones: \medskip
\begin{itemize}
\item[\bf (i)$^\prime$]  
$ |\partial^\alpha_\xi\partial^\beta_x a(t,x,\xi)|\leq C_\alpha C^{|\beta|+1}\beta!,\quad |\alpha|+|\beta|\geq 2,
$
\medskip
\item[\bf (iii)$^\prime$] 
$ |\partial^\alpha_\xi\partial^\beta_x b(t,x,\xi)|\leq C_\alpha C^{|\beta|+1}\beta!,\quad |\alpha|+|\beta|\geq 1,
$
\end{itemize}
for every $x,\xi\in\rd$, $t\in[0,T]$, and for constants $C_\alpha>0$ independent of $\beta$ and $C>0$ independent of $\alpha$ and $\beta$.\par
\begin{theorem}
Assume the hypotheses ${\bf (i)',(ii),(iii)',(iv)}$ for the operator $L$ in \eqref{equazione}. Then the corresponding linear propagator $S(t)$ maps $\mathcal A\to\mathcal A$. 
\end{theorem}
\begin{proof}
For $\epsilon>0$ and $N\in\bN$, define the following inner product in $H^N(\rd)$:
\[
\mathcal{E}^\epsilon_N[u,v]=\sum_{|\alpha|\leq N} \frac{\epsilon^{2|\alpha|}}{\alpha!^2}\langle \partial^\alpha u,\partial^\alpha v\rangle.
\]
Clearly, it suffices to prove that for $\epsilon>0$ small enough the following a priori estimate holds with a constant $C>0$ independent of $N$:
\[
\mathcal{E}^\epsilon_N[u(t),u(t)]^{1/2}\leq C \mathcal{E}^\epsilon_N[u(0),u(0)]^{1/2}+C\int_0^t \mathcal{E}^\epsilon_N[Lu(s),Lu(s)]^{1/2}\,ds \quad\forall t\in[0,T].
\]
This can be obtained as in the proof of Proposition \ref{proposizione} once we have the following lower bound for $\mathcal{L}:=a^w(t,x,D)+ib^w(t,x,D)$:  
\[
{\rm Re}\,\mathcal{E}^\epsilon_N[\mathcal{L}u(t),u(t)]\geq -C \mathcal{E}^\epsilon_N[u(t),u(t)]
\]
for some constant $C>0$ independent of $N$. This will be proved in the following lemma. 
\end{proof}

\begin{lemma}\label{lemma2} Let $a(x,\xi)$ and $b(x,\xi)$ be real-valued symbols satisfying the estimates
\begin{equation}\label{a0}
 |\partial^\alpha_\xi\partial^\beta_x a(x,\xi)|\leq C_\alpha C^{|\beta|+1}\beta!,\quad |\alpha|+|\beta|\geq 2
\end{equation}
\[ |\partial^\alpha_\xi\partial^\beta_x b(x,\xi)|\leq C_\alpha C^{|\beta|+1}\beta!,\quad |\alpha|+|\beta|\geq 1.
\]
 Assume, moreover, that $a(x,\xi)\geq -C$ for every $x,\xi\in\rd$ and some constant $C>0$.\par
Let $\mathcal{L}=a^w(x,D)+ib^w(x,D)$. There exist $\epsilon>0$, $C>0$ such that for every $N\in\bN$ and $u\in\cS(\rd)$,
\[
{\rm Re}\,\mathcal{E}^\epsilon_N[\mathcal{L}u,u]\geq -C\mathcal{E}^\epsilon_N[u,u].
\]

\end{lemma}
\begin{proof}
 It suffices to prove the estimates
\begin{equation}\label{lba}
{\rm Re}\,\mathcal{E}^\epsilon_N[a^w(x,D)u,u]\geq -C\mathcal{E}^\epsilon_N[u,u],\quad |{\rm Re}\, \mathcal{E}^\epsilon_N[i b^w(x,D)u,u]|\leq C\mathcal{E}^\epsilon_N[u,u].
\end{equation}
Let us prove the first estimate in \eqref{lba}. We have
\[
\langle\partial^\alpha a^w(x,D) u,\partial^\alpha u\rangle=\langle a^w(x,D) \partial^\alpha u,\partial^\alpha u\rangle+ \langle [\partial^\alpha,a^w(x,D)] u,\partial^\alpha u\rangle.
\]
The first term is estimated from below by the sharp G\a r arding inequality, and we are left to prove that 
\begin{equation}\label{a1}
\sum_{|\alpha|\leq N}\frac{\epsilon^{2|\alpha|}}{\alpha!^2}|{\rm Re}  \langle [\partial^\alpha,a^w(x,D)] u,\partial^\alpha u\rangle|\leq C\mathcal{E}^\epsilon_N[u,u].
\end{equation}
It is slightly easier to work with the usual (left) quantization; by the standard pseudodifferential calculus we have $a^w(x,D)=a(x,D)+r(x,D)$, where $r(x,\xi)$ satisfies the estimates 
\begin{equation}\label{a2}
|\partial^\alpha_\xi \partial^\beta_x r(x,\xi)|\leq C_\alpha C^{|\beta|+1}\beta!,\quad \alpha,\beta\in\bN^d.
\end{equation}
Now we have, by an explicit computation,
\[
[\partial^\alpha,a^w(x,D)]=\sum_{0\not=\beta\leq\alpha}\binom{\alpha}{\beta} \big( (\partial^\beta_x a)(x,D)+ (\partial^\beta_x r)(x,D)\big)\partial^{\alpha-\beta}
\]
and therefore
\begin{align*}
 \langle [\partial^\alpha,a^w(x,D)] u,\partial^\alpha u\rangle=&
\sum_{\beta\leq\alpha,\ |\beta|\geq 2} \binom{\alpha}{\beta}
 \langle(\partial^\beta_x a)(x,D)\partial^{\alpha-\beta} u,\partial^{\alpha} u\rangle \\
 &+\sum_{\beta\leq\alpha,\ |\beta|=1} \binom{\alpha}{\beta}
 \langle(\partial^\beta_x a)(x,D)\partial^{\alpha-\beta} u,\partial^{\alpha} u\rangle \\
&+ \sum_{0\not=\beta\leq\alpha} \binom{\alpha}{\beta}
 \langle(\partial^\beta_x r)(x,D)\partial^{\alpha-\beta} u,\partial^{\alpha} u\rangle.
 \end{align*}
Hence, \eqref{a1} will follow from the following three estimates:
\begin{equation}\label{a3}
\sum_{|\alpha|\leq N} \sum_{\beta\leq\alpha,|\beta|\geq 2} \frac{\epsilon^{2|\alpha|}}{\alpha!^2} \binom{\alpha}{\beta}
| \langle(\partial^\beta_x a)(x,D)\partial^{\alpha-\beta} u,\partial^{\alpha} u\rangle|\leq C\mathcal{E}^\epsilon_N[u,u],
\end{equation}
\begin{equation}\label{a4}
\sum_{|\alpha|\leq N} \sum_{\beta\leq\alpha,|\beta|= 1} \frac{\epsilon^{2|\alpha|}}{\alpha!^2} \binom{\alpha}{\beta}
|{\rm Re} \langle(\partial^\beta_x a)(x,D)\partial^{\alpha-\beta} u,\partial^{\alpha} u\rangle|\leq C\mathcal{E}^\epsilon_N[u,u],
\end{equation}
\begin{equation}\label{a5}
\sum_{|\alpha|\leq N} \sum_{0\not=\beta\leq\alpha} \frac{\epsilon^{2|\alpha|}}{\alpha!^2} \binom{\alpha}{\beta}
| \langle(\partial^\beta_x r)(x,D)\partial^{\alpha-\beta} u,\partial^{\alpha} u\rangle|\leq C\mathcal{E}^\epsilon_N[u,u].
\end{equation}
Let us prove \eqref{a3}. We observe that the operator $(\partial^\beta_x a)(x,D)$ is bounded on $L^2$ and its operator norm is estimated by a seminorm of its symbol in $S^0_{0,0}$ depending only on the dimension $d$; hence by \eqref{a0} we have, for some constant $C>0$, 
\begin{equation}\label{a6}
\|(\partial^\beta_x a)(x,D) u\|\leq C^{|\beta|+1}\beta!\|u\|,\quad |\beta|\geq 2.
\end{equation}
By the Cauchy-Schwarz' inequality in $L^2$ and \eqref{a6}, the left-hand side of \eqref{a3} is estimated by
\[
C\sum_{|\alpha|\leq N}\sum_{\beta\leq\alpha, |\beta|\geq 2} (C\epsilon)^{|\beta|}\frac{\epsilon^{|\alpha-\beta|}}{(\alpha-\beta)!}\|\partial^{\alpha-\beta} u\|\cdot \frac{\epsilon^{|\alpha|}}{\alpha!}\|\partial^{\alpha} u\|. 
\]
By applying the Cauchy-Schwarz' inequality and then the Minkowski's inequality for sequences we continue the estimate as 
\[
\leq C\sum_{2\leq |\beta|\leq N} (C\epsilon)^{|\beta|}\underbrace{\Big(\sum_{|\alpha|\leq N,\alpha\geq\beta} \frac{\epsilon^{2|\alpha-\beta|}}{(\alpha-\beta)!^2}\|\partial^{\alpha-\beta} u\|^2\Big)^{1/2}}_{\leq \mathcal{E}^\epsilon_N[u,u]^{1/2}}\mathcal{E}^\epsilon_N[u,u]^{1/2}.
\]
If $\epsilon<C^{-1}$, we have $\sum_{|\beta|\geq 2} (C\epsilon)^{|\beta|}<\infty$ and we get \eqref{a3}. \par
Consider now \eqref{a4}. We can write the left-hand side as
\[
\sum_{|\alpha|\leq N} \sum_{\beta\leq\alpha,|\beta|= 1} \frac{\epsilon^{2|\alpha|}}{\alpha!^2} \binom{\alpha}{\beta}
|{\rm Re} \langle \partial^\beta(\partial^\beta_x a)(x,D)\partial^{\alpha-\beta} u,\partial^{\alpha-\beta} u\rangle|.
\] 
Now, by the pseudodifferential calculus we have 
$$(\partial^\beta_x a)(x,D)^\ast= (\partial^\beta_x a)(x,D)+r'(x,D)$$
 with $r'$ satisfying the same estimates \eqref{a2}. By a direct computation we see that the self-adjoint part of the operator $\partial^\beta(\partial^\beta_x a)(x,D)$ is given by 
\begin{align*}
{\rm Re}\, \partial^\beta(\partial^\beta_x a)(x,D)&=\partial^\beta(\partial^\beta_x a)(x,D)-(\partial^\beta_x a)(x,D)^\ast \partial^\beta\\
&=  (\partial^{2\beta} a)(x,D)-\partial^\beta_x r'(x,D)+(\partial^\beta_x r')(x,D).
\end{align*}
Hence we can write the left-hand side of \eqref{a4} as 
\begin{align*}
&\ \sum_{|\alpha|\leq N} \sum_{\beta\leq\alpha,|\beta|= 1} \frac{\epsilon^{2|\alpha|}}{\alpha!^2} \binom{\alpha}{\beta}
|{\rm Re} \langle (\partial^{2\beta}_x a)(x,D)\partial^{\alpha-\beta} u,\partial^{\alpha-\beta} u\rangle|\\
&+\sum_{|\alpha|\leq N} \sum_{\beta\leq\alpha,|\beta|= 1} \frac{\epsilon^{2|\alpha|}}{\alpha!^2} \binom{\alpha}{\beta}
|{\rm Re} \langle r'(x,D)\partial^{\alpha-\beta} u,\partial^{\alpha} u\rangle|\\
&+\sum_{|\alpha|\leq N} \sum_{\beta\leq\alpha,|\beta|= 1} \frac{\epsilon^{2|\alpha|}}{\alpha!^2} \binom{\alpha}{\beta}
|{\rm Re} \langle (\partial^{\beta}_x r')(x,D)\partial^{\alpha-\beta} u,\partial^{\alpha-\beta} u\rangle|.
\end{align*}
These expressions can then be estimated as in the proof of \eqref{a3}. \par 
Similarly, we can verify \eqref{a5} exactly as we did for \eqref{a3} (using \eqref{a2}).\par
Finally, the second inequality in \eqref{lba} can be obtained by similiar (and easier) arguments. 
\end{proof}

\section{Global pseudolocality}\label{applicazioni}
Let us recall the definition of global wave front set from \cite{hormander}. Consider the symbol class $\Gamma^m$, $m\in\R$, of  smooth functions $p(x,\xi)$ in $\rdd$ satisfying the estimates
\[
|\partial^\alpha_\xi\partial^\beta_x p(x,\xi)|\leq C_{\alpha\beta}(1+|x|+|\xi|)^{m-|\alpha|-|\beta|},\quad \forall \alpha,\beta\in\bN^d,\ x,\xi\in\rd.
\]
A point $(x_0,\xi_0)$ is called non-characteristic for $p\in \Gamma^m$ if there are $\epsilon,C>0$ such that 
\[
|p(x,\xi)|\geq C(1+|x|+|\xi|)^m \quad{\rm for}\ (x,\xi)\in V_{(x_0,\xi_0),\epsilon}
\]
where, for $z_0\in\rdd$, $z_0\not=0$, $V_{z_0,\epsilon}$ is the conic neighborhood
\[
V_{z_0,\epsilon}=\Big\{z\in\rdd\setminus\{0\}:\,\Big| \frac{z}{|z|}-\frac{z_0}{|z_0|} \Big|<\epsilon,\ |z|>\epsilon^{-1}\Big\} .
\]
 Let now $ f\in\cS'(\rd)$. We define its global wave front set $WF_G(f)\subset\rdd\setminus\{0\}$ by saying that $(x_0,\xi_0)\in\rd\times\rd$, $(x_0,\xi_0)\not=0$, does not belong to $WF_G(f)$ if there exists $\psi\in \Gamma^0$ which is non-characteristic at $(x_0,\xi_0)$, such that $\psi(x,D)f\in\cS(\rd)$. The set $WF_G(f)$ is a closed conic subset of $\rdd\setminus\{0\}$. As observed in the introduction and proved in \cite{hormander},  if $f\in\cS'(\rd)$ then $f\in\cS(\rd)$ if and only if $WF_G(f)=\emptyset$. \par
Now, we know from \cite[Theorem 4.1]{rw} that pseudodifferential operators with symbols in $S^0_{0,0}$ decreases such a global wave front set. Hence from Theorem \ref{mainteo} we get at once the following result for the linear propagator $S(t)$. 
\begin{corollary}
With the above notation, $WF_G(S(t)f)\subset WF_G(f)$ for every $f\in\cS'(\rd)$, $t\in[0,T]$. 
\end{corollary}
The statement makes sense because $S(t)$, as a $S^0_{0,0}$-type pseudodifferential operator, extends to a bounded operator $\cS'(\rd)\to\cS'(\rd)$.\par

\section*{Acknowledgements}
It is a pleasure to express my gratitude to Elena Cordero, Luigi Rodino and Emanuela Sasso for interesting discussions on the subject of this paper.

\end{document}